\documentclass{amsart}
\usepackage{amsfonts}

\newtheorem{theorem}{Theorem}
\theoremstyle{plain}

\newtheorem{definition}{Definition}
\newtheorem{example}{Example}

\newtheorem{lemma}{Lemma}

\newtheorem{remark}{Remark}

\numberwithin{equation}{section}

\input{tcilatex}

\begin{document}
\title[$F$-harmonic maps as global maxima]{$F$-harmonic maps as global maxima%
}
\author{Mohammed Benalili}
\address{Dept. of Mathematics, Facult\'{e} des Sciences, Universit\'{e}
Abou-Belka\"{\i}d, Tlemcen}
\email{m\_benalili@mail.univ-tlemcen.dz}
\author{Hafida Benallal}
\email{h\_benallal@mail.univ-tlemcen.dz}
\subjclass[2000]{Primary 58E20.}
\keywords{F-harmonic maps, Stress-energy tensor.}

\begin{abstract}
In this note, we show that some $F$-harmonic maps into spheres are global
maxima of the variations of their energy functional on the conformal group
of the sphere. Our result extends partially those obtained in \cite{15} and 
\cite{17} for harmonic and $p$-harmonic maps.
\end{abstract}

\maketitle

\section{\protect\bigskip Introduction}

Harmonic maps have been studied first by J. Eells and J.H.Sampson in the
sixties and since then many articles have appeared ( see \cite{6}, \cite{12}%
, \cite{16}, \cite{19}, \cite{20}, \cite{24}) to cite a few of them.
Extensions to the notions of $p$-harmonic, biharmonic, $F$-harmonic and $f$%
-harmonic maps were introduced and similar research has been carried out
(see \cite{1}, \cite{2}, \cite{3}, \cite{7}, \cite{15}, \cite{18}, \cite{21}%
, \cite{23}). Harmonic maps were applied to broad areas in sciences and
engineering including the robot mechanics ( see \cite{5}, \cite{8}, \cite{9}
)$.$The concept of $F$- harmonic maps unifies the notions of harmonic maps, $%
p$-harmonic maps, minimal hypersurfaces. An important tool for studying
stability of stability of $F$ harmonic maps is the stress-energy tensor.

In this paper for a $C^{2}$-function $F:\left[ 0,+\infty \right[ \rightarrow %
\left[ 0,+\infty \right[ $ such that $F^{\prime }(t)>0$ on $t\in \left]
0,+\infty \right[ $, we look for sufficient conditions which present $F$%
-harmonic maps into spheres as global maxima of the energy functional. Our
result extends similar results obtained in \cite{17} and \cite{18} for
harmonic and $p$-harmonic maps.

Let $(M,g)$ and $S^{n\text{ }}$be, respectively, a compact Riemannian
manifold of dimension $m\geq 2$ and $\ $the unit $n$-dimensional Euclidean
sphere with $n\geq 2$ endowed with the canonical metric $can$ induced by the
inner product of $R^{n+1}$.

For a $C^{1}$- application $\phi :(M,g)\longrightarrow (S^{n},can)$, we
define the $F$-energy functional by 
\begin{equation*}
E_{F}(\phi )=\int_{M}F\left( \frac{\left\vert d\phi \right\vert ^{2}}{2}%
\right) dv_{g}\text{,}
\end{equation*}%
where $\frac{\left\vert d\phi \right\vert ^{2}}{2}$ denotes the energy
density given by 
\begin{equation*}
\frac{\left\vert d\phi \right\vert ^{2}}{2}=\frac{1}{2}\sum_{i=1}^{m}\left%
\vert d\phi (e_{i})\right\vert ^{2}
\end{equation*}%
and where $\left\{ e_{i}\right\} $ is an orthonormal basis on the tangent
space $T_{x}M$ and $dv_{g}$ is the Riemannian measure associated to $g$ on $%
M $.

Let $\phi ^{-1}TS^{n}$and $\Gamma \left( \phi ^{-1}TS^{n}\right) $ be,
respectively, the pullback vector fiber bundle of $TS^{n}$ and the space of
sections on $\phi ^{-1}TS^{n}$. Denote by $\nabla ^{M}$, $\nabla ^{S^{n}}$%
and $\nabla $, respectively, the Levi-Civita connections on: $TM$, $T$ $%
S^{n} $ and $\phi ^{-1}TS^{n}$. Recall that $\nabla $ is defined by \ 
\begin{equation*}
\nabla _{X}Y=\nabla _{\phi _{\ast }X}^{S^{n}}Y
\end{equation*}%
where $X\in TM$ and $Y\in \Gamma \left( \phi ^{-1}TS^{n}\right) $.

Let $v$ be a vector field on $S^{n}$ and denote by $\left( \gamma
_{t}^{v}\right) _{t}$ the flow of diffeomorphisms induced by $v$ on $S^{n}$
i.e. 
\begin{equation*}
\gamma _{0}^{v}=id\ \text{, \ \ }\frac{d}{dt}\gamma _{t}^{v}{}_{t=0}=v\left(
\gamma _{t}^{v}\right) \text{.}
\end{equation*}%
Denote by $\phi _{t}=\gamma _{t}^{v}o\phi $ the flow generated by $v$ along
the map $\phi $. The first variation formula of $E_{F}(\phi )$ is given by%
\begin{equation*}
\frac{d}{dt}E_{F}(\phi _{t})\mid _{t=0}=\int_{M}F^{\prime }\left( \frac{%
\left\vert d\phi _{t}\right\vert ^{2}}{2}\right) \left\langle \nabla
_{\partial t}d\phi _{t},d\phi _{t}\right\rangle \left\vert _{t=0}\right.
dv_{g}
\end{equation*}%
\begin{equation*}
=-\int_{M}\left\langle v,\tau _{F}(\phi )\right\rangle dv_{g}
\end{equation*}%
where $\tau _{F}(\phi )=trace_{g}\nabla \left( F^{\prime }\left( \frac{%
\left\vert d\phi \right\vert ^{2}}{2}\right) d\phi \right) $ is the $F$%
-tension.

\begin{definition}
$\phi $ is said $F$-harmonic if and only if $\tau _{F}(\phi )=0$ i.e. $\phi $
is a critical point of $\ $the $F$-energy functional $E_{F}$.
\end{definition}

Let $v\in 
%TCIMACRO{\U{211d} }%
%BeginExpansion
\mathbb{R}
%EndExpansion
^{n+1}$ and set $\bar{v}(y)=v-\left\langle v,y\right\rangle y$ \ for any $%
y\in S^{n}$. It is known that $\bar{v}$ is a conformal vector field on $%
S^{n} $ i.e. $\left( \gamma _{t}^{v}\right) ^{\ast }can=\alpha _{t}^{2}can$
where $\left( \gamma _{t}^{v}\right) _{t}$ denotes the flow induced by the
vector field $\bar{v}$. The expression of $\alpha _{t}$ is given in \cite{17}
by 
\begin{equation}
\alpha _{t}=\frac{\left\vert v\right\vert }{\left\vert v\right\vert cht+\phi
_{v}sht}\text{.}  \label{1}
\end{equation}%
where $\phi _{v}\left( x\right) =\left\langle v,\phi \left( x\right)
\right\rangle $ and $\left\langle .,.\right\rangle $ the inner product on
the Euclidean space $%
%TCIMACRO{\U{211d} }%
%BeginExpansion
\mathbb{R}
%EndExpansion
^{n+1}$. Denote by $\pounds (\phi )$ the subspace of $\Gamma (\phi
^{-1}TS^{n})$ given by 
\begin{equation*}
\pounds (\phi )=\left\{ \bar{v}\circ \phi ,v\in 
%TCIMACRO{\U{211d} }%
%BeginExpansion
\mathbb{R}
%EndExpansion
^{n+1}\right\} \text{.}
\end{equation*}%
Obviously, if $\phi $ is not constant, $\pounds (\phi )$ is of dimension $%
n+1 $.

\section{ $F$-harmonic maps as global maxima}

For any $\overline{v}\in \pounds (\phi )$, we denote by $\left( \gamma
_{t}^{v}\right) _{t\in R}$ the one parameter group of conformal
diffeomorphisms on $S^{n}$ induced by the vector $\bar{v}$. For\ a $C^{2}$%
-function $F:\left[ 0,+\infty \right[ \rightarrow \left[ 0,+\infty \right[ $
such that $F^{\prime }(t)>0$ in $\left] 0,+\infty \right[ $.

Now we introduce the following tensor field 
\begin{equation*}
S_{g}^{F}\left( \phi \right) =F^{\prime }\left( \frac{\left\vert d\phi
\right\vert ^{2}}{2}\right) \frac{\left\vert d\phi \right\vert ^{2}}{2}g
\end{equation*}%
\begin{equation*}
-\left[ F^{\prime }\left( \frac{\left\vert d\phi \right\vert ^{2}}{2}\right)
+\frac{\left\vert d\phi \right\vert ^{2}}{2}F^{\prime \prime }\left( \frac{%
\left\vert d\phi \right\vert ^{2}}{2}\right) \right] \phi ^{\ast }can.
\end{equation*}%
For $x\in M$, we set 
\begin{equation*}
S_{g}^{o,F}(\phi )\left( x\right) =\inf \left\{ S_{g,v}^{F}(\phi )(X,X)\text{%
, }X\in T_{x}M\text{ such that }g(X,X)=1\right\} \text{.}
\end{equation*}%
The tensor field $S_{g}^{F}(\phi )\left( x\right) $ will be said positive (
resp. positive defined) at $x$ if $\ S_{g}^{o,F}(\phi )\left( x\right) \geq
0 $ (resp.$\ \ S_{g}^{o,F}(\phi )\left( x\right) >0$). The tensor field $%
S_{g}^{F}(\phi )$ will be called the $F$ stress-energy tensor of $\phi $.
The tensor field $S_{g}^{F}(\phi )$ is different from the one defined by Ara
given by $S_{F}\left( \phi \right) =F^{{}}\left( \frac{\left\vert d\phi
\right\vert ^{2}}{2}\right) g-F^{^{\prime }}\left( \frac{\left\vert d\phi
\right\vert ^{2}}{2}\right) \phi ^{\ast }can$, but $S_{g}^{F}(\phi )$ is
more suitable for our case.

\begin{example}
For $F(t)=\frac{1}{p}\left( 2t\right) ^{\frac{p}{2}}$, with $p=2$ or $p\geq 4
$, $S_{g}^{p}\left( \phi \right) $ is the stress-energy tensor introduced,
respectively, by Eells and Lemaire for $p=2$ \cite{12} and modulo a
multiplied positive constant by El Soufi for $p\geq 4$ \cite{16}, so we may
call $S_{g}^{F}(\phi )$ the stress-energy tensor of $\phi $.
\end{example}

Indeed if $F(t)=t$ then $F\prime (t)=1$, $F^{\prime \prime }(t)=0$ and%
\begin{equation*}
S_{g}^{F}\left( \phi \right) =\frac{\left\vert d\phi \right\vert ^{2}}{2}%
g-\phi ^{\ast }can.
\end{equation*}%
In the case $F(t)=\frac{1}{p}\left( 2t\right) ^{\frac{p}{2}}$, with $p\geq 4$%
, $F\prime (t)=(2t)^{\frac{p}{2}-1}$, $\ F^{\prime \prime }(t)=\left(
p-2\right) \left( 2t\right) ^{\frac{p}{2}-2}$ and%
\begin{equation*}
S_{g}^{F}\left( \phi \right) =\frac{1}{2}\left\vert d\phi \right\vert ^{p}g-%
\frac{p}{2}\left\vert d\phi \right\vert ^{p-2}\phi ^{\ast }can=\frac{p}{2}%
\left( \frac{1}{p}\left\vert d\phi \right\vert ^{p}g-\left\vert d\phi
\right\vert ^{p-2}\phi ^{\ast }can\right)
\end{equation*}%
The function $F$ is called \textit{admissible} if it satisfies 
\begin{equation*}
B=\left( \frac{F^{\prime \prime }(\alpha _{t}^{2}o\phi .\frac{\left\vert
d\phi \right\vert ^{2}}{2})}{F^{\prime }(\alpha _{t}^{2}o\phi .\frac{%
\left\vert d\phi \right\vert ^{2}}{2})}\alpha _{t}^{2}o\phi -\frac{F^{\prime
\prime }\left( \frac{\left\vert d\phi \right\vert ^{2}}{2}\right) }{%
F^{\prime }\left( \frac{\left\vert d\phi \right\vert ^{2}}{2}\right) }%
\right) \phi _{v}\geq 0
\end{equation*}%
and the $F$ stress-energy tensor $S_{g}^{F}\left( \phi \right) $ of $\phi $
fulfills 
\begin{equation*}
S_{g}^{F}\left( \gamma _{t}o\phi \right) \geq \theta \left( \alpha
_{t}^{2}o\phi \right) .S_{g}^{F}\left( \phi \right) \text{ }
\end{equation*}%
where $\theta $ is a real positive function and $\gamma _{t}$ is the one
parameter group of conformal transformations induced by the vector field $%
\overline{v}$ ( defined above ) on the euclidean sphere $S^{n}$ and $\alpha
_{t}$ is given by (\ref{1}).

\begin{example}
\end{example}

The function $F(t)=\frac{1}{p}\left( 2t\right) ^{\frac{p}{2}}$ for $p=2$ and 
$p\geq 4$ and $t\geq 0$ is admissible.

Indeed, \ for $F(t)=\frac{1}{p}\left( 2t\right) ^{\frac{p}{2}}$ we have $B=0$
and for any conformal diffeomorphism $\gamma $ on the euclidean sphere, we
have 
\begin{equation*}
S_{g}^{F}\left( \gamma o\phi \right) =\frac{1}{2}\left\vert d\left( \gamma
_{t}o\phi \right) \right\vert ^{p}g-\frac{p}{2}\left\vert d\left( \gamma
_{t}o\phi \right) \right\vert ^{p-2}\phi ^{\ast }can
\end{equation*}%
so if we let $\left\vert d\left( \gamma _{t}o\phi \right) \right\vert
^{2}=\alpha _{t}^{2}o\phi .\left\vert d\phi \right\vert ^{2}$, we get%
\begin{equation*}
S_{g}^{F}\left( \gamma _{t}o\phi \right) =\alpha _{t}^{2}o\phi .\left( \frac{%
1}{2}\left\vert d\phi \right\vert ^{p}g-\frac{p}{2}\left\vert d\phi
\right\vert ^{p-2}\phi ^{\ast }can\right)
\end{equation*}%
\begin{equation*}
=\alpha _{t}^{2}o\phi .S_{g}^{F}\left( \phi \right) \text{.}
\end{equation*}%
The $F(t)=1+at-e^{-t}$, for $t\in \left[ 0,+\infty \right[ $ where $%
a=\max_{x\in M}\frac{\left\vert d\phi \right\vert ^{2}}{2}$ is admissible
provided that the conformal diffeomorphism on the euclidean sphere $S^{n}$
is contracting that means that the function $\phi _{v}$ given in the
expression of (\ref{1}) is nonnegative and the stress-energy tensor $%
S_{g}\left( \phi \right) =\frac{\left\vert d\phi \right\vert ^{2}}{2}g-\phi
^{\ast }can$ of $\phi $ is positive.

Indeed, we have 
\begin{equation*}
B=\left( -\alpha _{t}^{2}o\phi \frac{e^{-\alpha _{t}^{2}o\phi \frac{%
\left\vert d\phi \right\vert ^{2}}{2}}}{a+e^{-\alpha _{t}^{2}o\phi \frac{%
\left\vert d\phi \right\vert ^{2}}{2}}}+\frac{e^{-\frac{\left\vert d\phi
\right\vert ^{2}}{2}}}{a+e^{-\frac{\left\vert d\phi \right\vert ^{2}}{2}}}%
\right) \phi _{v}
\end{equation*}%
Putting $u=\alpha _{t}^{2}o\phi \in \left] 0,1\right] $, we consider the
function $\varphi \left( u\right) =-$.$u\frac{e^{-u\frac{\left\vert d\phi
\right\vert ^{2}}{2}}}{a+e^{-u\frac{\left\vert d\phi \right\vert ^{2}}{2}}}+%
\frac{e^{-\frac{\left\vert d\phi \right\vert ^{2}}{2}}}{a+e^{-\frac{%
\left\vert d\phi \right\vert ^{2}}{2}}},$we get%
\begin{equation*}
\varphi ^{\prime }\left( u\right) =\left( \frac{\left\vert d\phi \right\vert
^{2}}{2}u-a-e^{-\left\vert d\phi \right\vert ^{2}u}\right) \frac{e^{-\frac{%
\left\vert d\phi \right\vert ^{2}}{2}}}{\left( a+e^{-\frac{\left\vert d\phi
\right\vert ^{2}}{2}u}\right) ^{2}}
\end{equation*}%
and it is obvious that $\varphi ^{\prime }(u)\leq 0$, hence $\varphi $ is a
decreasing function on $\left] 0,1\right] $ i.e. $\varphi (u)\geq \varphi
(1)=0$. Consequently $B\geq 0$.

Now%
\begin{equation*}
S_{g}^{F}\left( \gamma _{t}o\phi \right) \left( X;X\right) =\left( a+e^{-%
\frac{\left\vert d\left( \gamma _{t}o\phi \right) \right\vert ^{2}}{2}%
}\right) \frac{\left\vert d\left( \gamma _{t}o\phi \right) \right\vert ^{2}}{%
2}g\left( X,X\right)
\end{equation*}%
\begin{equation*}
-\left[ \left( a+e^{-\frac{\left\vert d\left( \gamma _{t}o\phi \right)
\right\vert ^{2}}{2}}\right) -\frac{\left\vert d\left( \gamma _{t}o\phi
\right) \right\vert ^{2}}{2}e^{-\frac{\left\vert d\left( \gamma _{t}o\phi
\right) \right\vert ^{2}}{2}}\right] \left( \gamma _{t}o\phi \right) ^{\ast
}can\left( X,X\right) .
\end{equation*}%
\begin{equation*}
=\alpha _{t}^{2}o\phi \left( a+e^{-\alpha _{t}^{2}o\phi \frac{\left\vert
d\phi \right\vert ^{2}}{2}}\right) \frac{\left\vert d\phi \right\vert ^{2}}{2%
}g\left( X,X\right)
\end{equation*}%
\begin{equation*}
-\alpha _{t}^{2}o\phi \left[ \left( a+e^{-\alpha _{t}^{2}o\phi \frac{%
\left\vert d\phi \right\vert ^{2}}{2}}\right) -\alpha _{t}^{2}o\phi \frac{%
\left\vert d\phi \right\vert ^{2}}{2}e^{-\alpha _{t}^{2}o\phi \frac{%
\left\vert d\phi \right\vert ^{2}}{2}}\right] \phi ^{\ast }can\left(
X,X\right)
\end{equation*}%
\begin{equation*}
=\alpha _{t}^{2}o\phi \left( a+e^{-\frac{\left\vert d\left( \gamma _{t}o\phi
\right) \right\vert ^{2}}{2}}\right) \left[ \frac{1}{2}\left\vert d\phi
\right\vert ^{2}g\left( X,X\right) -\phi ^{\ast }can\left( X,X\right) \right]
\end{equation*}%
\begin{equation*}
+\alpha _{t}^{2}o\phi \frac{\left\vert d\phi \right\vert ^{2}}{2}e^{-\alpha
_{t}^{2}o\phi \frac{\left\vert d\phi \right\vert ^{2}}{2}}\phi ^{\ast
}can\left( X,X\right) .
\end{equation*}%
An other example is the following function $F(t)=\left( 1+2t\right) ^{\alpha
}$ where $0<\alpha <1$, the $F$-energy is the $\alpha $-energy of
Sacks-Uhlenbeck ( see \cite{22}). In fact 
\begin{equation*}
B=\left( \alpha -1\right) \left( \frac{1}{1+\alpha _{t}^{2}o\phi .\left\vert
d\phi \right\vert ^{2}}\alpha _{t}^{2}o\phi -\frac{1}{1+\left\vert d\phi
\right\vert ^{2}}\right) \phi _{v}
\end{equation*}%
\begin{equation*}
=\frac{\left( \alpha -1\right) \left( \alpha _{t}^{2}o\phi -1\right) }{%
\left( 1+\alpha _{t}^{2}o\phi .\left\vert d\phi \right\vert ^{2}\right)
\left( 1+\left\vert d\phi \right\vert ^{2}\right) }\phi _{v}\geq 0
\end{equation*}%
provided that $\phi _{v}\geq 0$.

And for vector field $X$ on $M$, we have 
\begin{equation*}
S_{g}^{F}\left( \gamma _{t}o\phi \right) \left( X;X\right) =2\alpha \left(
1+\alpha _{t}^{2}o\phi \left\vert d\phi \right\vert ^{2}\right) ^{\alpha
-1}\alpha _{t}^{2}o\phi \frac{\left\vert d\phi \right\vert ^{2}}{2}g\left(
X,X\right) -
\end{equation*}%
\begin{equation*}
\left[ 2\alpha \left( 1+\alpha _{t}^{2}o\phi \left\vert d\phi \right\vert
^{2}\right) ^{\alpha -1}+2\alpha \left( \alpha -1\right) \left( 1+\alpha
_{t}^{2}o\phi \left\vert d\phi \right\vert ^{2}\right) ^{\alpha -2}\alpha
_{t}^{2}o\phi .\left\vert d\phi \right\vert ^{2}\right] \alpha _{t}^{2}o\phi
.\phi ^{\ast }can\left( X,X\right)
\end{equation*}%
\begin{equation*}
=\left[ 2\alpha \left( 1+\alpha _{t}^{2}o\phi \left\vert d\phi \right\vert
^{2}\right) ^{\alpha -1}\alpha _{t}^{2}o\phi \left( \frac{\left\vert d\phi
\right\vert ^{2}}{2}g\left( X,X\right) -\phi ^{\ast }can\left( X,X\right)
\right) \right. +
\end{equation*}%
\begin{equation*}
\left. 2\alpha \left( 1-\alpha \right) \left( 1+\alpha _{t}^{2}o\phi
\left\vert d\phi \right\vert ^{2}\right) ^{\alpha -2}\alpha _{t}^{2}o\phi
\left\vert d\phi \right\vert ^{2}.\phi ^{\ast }can\left( X,X\right) \right]
\end{equation*}%
and taking account of the positivity of the stress-energy tensor of $%
S_{g}\left( \phi \right) =\frac{1}{2}\left\vert d\phi \right\vert ^{2}g-\phi
^{\ast }can$ and the fact that $\phi _{v}\geq 0$, we infer that 
\begin{equation*}
S_{g}^{F}\left( \gamma _{t}o\phi \right) \left( X,X\right) \geq \alpha
_{t}^{4}o\phi .S_{g}^{F}\phi \left( X,X\right) \text{.}
\end{equation*}

\begin{remark}
$\phi _{v}\geq 0$ occurs for example if $\phi (M)$ is included in the
positive half-sphere $S^{n+}=\left\{ x\in S^{n}:\left\langle
x,v\right\rangle \geq 0\right\} $.
\end{remark}

In this section we state the following result

\begin{theorem}
\label{th2} Let $F:\left[ 0,+\infty \right[ \rightarrow \left[ 0,+\infty %
\right[ $ be an admissible function and $\phi $ be an $F$-harmonic map from
a compact $m$-Riemannian manifold $\left( M,g\right) $ ($m\geq 2$) into the
Euclidean sphere $S^{n}$ ($n\geq 2$). Suppose that the $F$ stress- energy
tensor $S_{g}^{F}\left( \phi \right) $ is positive ( resp. positively
defined). Then for any conformal diffeomorphism $\gamma $ on $S^{n}$, $%
E_{F}\left( \gamma o\phi \right) \leq E_{F}\left( \phi \right) $ ( resp. $%
E_{F}\left( \gamma o\phi \right) <E_{F}\left( \phi \right) $ ).
\end{theorem}

\begin{remark}
In case $F(t)=\frac{1}{p}\left( 2t\right) ^{\frac{p}{2}}$, $p=2$ or $p\geq 4$
the condition $\phi _{v}\geq 0$ is not needed since $B=0$, so our result
recover the ones by El-Soufi in \cite{16} and \cite{18}.
\end{remark}

To prove Theorem \ref{th2}, we need the following lemmas

\begin{lemma}
\label{lem1} Let $\phi :(M,g)\rightarrow \left( N,h\right) $ be a smooth map
and $\gamma $ be a conformal diffeomorphism on $N$, then the $F$-tension of
the map $\gamma o\phi $, is given by%
\begin{equation*}
\tau _{F}\left( \gamma o\phi \right) =2\alpha ^{-1}o\phi .F^{\prime }(\alpha
^{2}o\phi \frac{\left\vert d\phi \right\vert ^{2}}{2})d\gamma \left( d\phi
_{v}-\frac{\left\vert d\phi \right\vert ^{2}}{2}\nabla \alpha o\phi \right)
\end{equation*}%
\begin{equation*}
+fd\gamma \left( \tau _{F}\left( \phi \right) \right) +d\gamma \left(
F^{\prime }\left( \frac{\left\vert d\phi \right\vert ^{2}}{2}\right) d\phi
\left( \nabla f\right) \right) \text{.}
\end{equation*}%
where $f=\frac{F^{\prime }\left( \alpha ^{2}o\phi \frac{\left\vert d\phi
\right\vert ^{2}}{2}\right) }{F^{\prime }\left( \frac{\left\vert d\phi
\right\vert ^{2}}{2}\right) }$ and $\gamma ^{\ast }can=\alpha ^{2}can$.
\end{lemma}

\begin{proof}
We follow closely the proof in \cite{18}%
\begin{equation*}
\tau _{F}\left( \gamma o\phi \right) =trace_{g}\nabla \left( F^{\prime
}\left( \frac{\left\vert d\left( \gamma o\phi \right) \right\vert ^{2}}{2}%
\right) d\left( \gamma o\phi \right) \right) =F^{\prime }\left( \frac{%
\left\vert d\left( \gamma o\phi \right) \right\vert ^{2}}{2}\right)
trace\left( \nabla d\left( \gamma o\phi \right) \right)
\end{equation*}%
\begin{equation*}
+d\left( \gamma o\phi \right) \left( \nabla F^{\prime }\left( \frac{%
\left\vert d\left( \gamma o\phi \right) \right\vert ^{2}}{2}\right) \right)
\end{equation*}%
where $\nabla F^{\prime }\left( \frac{\left\vert d\left( \gamma o\phi
\right) \right\vert ^{2}}{2}\right) $ is the gradient of $F^{\prime }\left( 
\frac{\left\vert d\left( \gamma o\phi \right) \right\vert ^{2}}{2}\right) $
in $M$.

Since $\gamma $ is a conformal diffeomorphism on $S^{n}$, we have 
\begin{equation*}
\tau _{F}\left( \gamma o\phi \right) =F^{\prime }\left( \alpha ^{2}o\phi 
\frac{\left\vert d\phi \right\vert ^{2}}{2}\right) \tau \left( \gamma o\phi
\right) +d\left( \gamma o\phi \right) \left( \nabla F^{\prime }\left( \alpha
^{2}o\phi \frac{\left\vert d\phi \right\vert ^{2}}{2}\right) \right)
\end{equation*}%
\begin{equation*}
=F^{\prime }\left( \alpha ^{2}o\phi \frac{\left\vert d\phi \right\vert ^{2}}{%
2}\right) \left( trace_{g}\nabla ^{\gamma }d\gamma \left( d\phi ,d\phi
\right) +d\gamma .\tau \left( \phi \right) \right) +d\left( \gamma o\phi
\right) \left( \nabla F^{\prime }\left( \alpha ^{2}o\phi \frac{\left\vert
d\phi \right\vert ^{2}}{2}\right) \right)
\end{equation*}%
\begin{equation*}
=F^{\prime }\left( \alpha ^{2}o\phi \frac{\left\vert d\phi \right\vert ^{2}}{%
2}\right) \left( trace_{g}\nabla ^{\gamma }d\gamma \left( d\phi ,d\phi
\right) +\frac{1}{F^{\prime }\left( \frac{\left\vert d\phi \right\vert ^{2}}{%
2}\right) }d\gamma \left( \tau _{F}\left( \phi \right) \right) \right)
\end{equation*}%
\begin{equation*}
-\frac{F^{\prime }(\alpha ^{2}o\phi \frac{\left\vert d\phi \right\vert ^{2}}{%
2})}{F^{\prime }\left( \frac{\left\vert d\phi \right\vert ^{2}}{2}\right) }%
d\left( \gamma o\phi \right) \left( \nabla F^{\prime }\left( \frac{%
\left\vert d\phi \right\vert ^{2}}{2}\right) \right) +d\left( \gamma o\phi
\right) \left( \nabla F^{\prime }\left( \alpha ^{2}o\phi \frac{\left\vert
d\phi \right\vert ^{2}}{2}\right) \right) \text{.}
\end{equation*}%
Putting $\ f=\frac{F^{\prime }(\alpha ^{2}o\phi \frac{\left\vert d\phi
\right\vert ^{2}}{2})}{F^{\prime }\left( \frac{\left\vert d\phi \right\vert
^{2}}{2}\right) }$ we get 
\begin{equation*}
\tau _{F}\left( \gamma o\phi \right) =F^{\prime }(\alpha ^{2}o\phi \frac{%
\left\vert d\phi \right\vert ^{2}}{2})trace_{g}\nabla ^{\gamma }d\gamma
\left( d\phi ,d\phi \right) +fd\gamma \left( \tau _{F}\left( \phi \right)
\right)
\end{equation*}%
\begin{equation*}
+F^{\prime }\left( \frac{\left\vert d\phi \right\vert ^{2}}{2}\right)
d\left( \gamma o\phi \right) \left( \nabla f\right) \text{.}
\end{equation*}

Now since $\gamma :\left( N,\gamma ^{\ast }can\right) \rightarrow \left(
N,can\right) $ is an isometry then, if $\ \widetilde{\nabla }$ denotes the
connection corresponding to $\gamma ^{\ast }can$, we have 
\begin{equation*}
\nabla ^{\gamma }d\gamma (X,Y)=d\gamma \left( \widetilde{\nabla }%
_{X}Y-\nabla _{X}Y\right)
\end{equation*}%
and since ( see \cite{18}) 
\begin{equation*}
\widetilde{\nabla }_{X}Y-\nabla _{X}Y=\alpha ^{-1}\left( \left\langle
X,\nabla \alpha \right\rangle Y+\left\langle Y,\nabla \alpha \right\rangle
X-\left\langle X,Y\right\rangle \nabla \alpha \right)
\end{equation*}%
we obtain%
\begin{equation*}
trace_{g}\nabla ^{\gamma }d\gamma \left( d\phi ,d\phi \right) =2\alpha
^{-1}o\phi .d\gamma \left( d\phi \left( \nabla \alpha o\phi \right) -\frac{%
\left\vert d\phi \right\vert ^{2}}{2}\nabla \alpha o\phi \right) \text{.}
\end{equation*}

Finally we infer that%
\begin{equation*}
\tau _{F}\left( \gamma o\phi \right) =2\alpha ^{-1}o\phi .F^{\prime }(\alpha
^{2}o\phi \frac{\left\vert d\phi \right\vert ^{2}}{2})d\gamma \left( d\phi
\left( \nabla \alpha o\phi \right) -\frac{\left\vert d\phi \right\vert ^{2}}{%
2}\nabla \alpha o\phi \right)
\end{equation*}%
\begin{equation*}
+fd\gamma \left( \tau _{F}\left( \phi \right) \right) +F^{\prime }\left( 
\frac{\left\vert d\phi \right\vert ^{2}}{2}\right) d\gamma od\phi \left(
\nabla f\right) \text{.}
\end{equation*}
\end{proof}

\begin{lemma}
\label{lem2} Let $\phi $ be an $F$-harmonic map from an $m$-dimensional
Riemannian manifold $\left( M,g\right) $ ($m\geq 2$) into the Euclidean unit
sphere $\left( S^{n},can\right) $ ($n\geq 2$).

Then for any $v\in R^{n+1}-\left\{ 0\right\} $ and any $t_{o}\in R$ we have%
\begin{equation*}
\frac{d}{dt}E_{F}\left( \gamma _{t}^{v}o\phi \right) _{t=t_{o}}=
\end{equation*}%
\begin{equation*}
-2\frac{sht_{o}}{\left\vert v\right\vert }\int_{M}\alpha _{t_{o}}^{3}o\phi
.F^{\prime }\left( \alpha _{t_{o}}^{2}o\phi .\frac{\left\vert d\phi
\right\vert ^{2}}{2}\right) \left( \left\vert d\phi \right\vert
^{2}\left\vert \overline{v}o\phi \right\vert ^{2}-\left\vert d\phi
_{v}\right\vert ^{2}\right) dv_{g}
\end{equation*}%
\begin{equation*}
-\int_{M}\alpha _{t_{o}}^{2}o\phi .F^{\prime }\left( \frac{\left\vert d\phi
\right\vert ^{2}}{2}\right) \left\langle d\phi \left( \nabla
f_{t_{o}}\right) ,\overline{v}o\phi \right\rangle dv_{g}
\end{equation*}%
where $f_{t_{o}}=\frac{F^{\prime }\left( \alpha _{t_{o}}^{2}o\phi \frac{%
\left\vert d\phi \right\vert ^{2}}{2}\right) }{F^{\prime }\left( \frac{%
\left\vert d\phi \right\vert ^{2}}{2}\right) }$ , $\gamma ^{\ast }can=\alpha
_{t_{o}}^{2}can$

and 
\begin{equation*}
\alpha _{t_{o}}=\frac{\left\vert v\right\vert }{\phi _{v}sht_{o}+\left\vert
v\right\vert cht_{o}}\text{.}
\end{equation*}
\end{lemma}

\begin{proof}
Recall that the first variation formula of the $F$-energy is given by%
\begin{equation*}
\frac{d}{dt}E_{F}\left( \gamma _{t}^{v}o\phi \right)
_{t=t_{o}}=-\int_{M}\left\langle \tau _{F}\left( \gamma _{t_{o}}^{v}o\phi
\right) ,\overline{v}o\left( \gamma _{to}^{v}o\phi \right) \right\rangle
dv_{g}\text{.}
\end{equation*}%
By Lemma \ref{lem1} and the fact that $\phi $ is $F$-harmonic we get 
\begin{equation*}
\frac{d}{dt}E_{F}\left( \gamma _{t}^{v}o\phi \right) _{t=t_{o}}=
\end{equation*}%
\begin{equation*}
-\int_{M}2\alpha _{t_{o}}o\phi .F^{\prime }\left( \alpha _{t_{o}}^{2}o\phi .%
\frac{\left\vert d\phi \right\vert ^{2}}{2}\right) \left\langle \left(
\nabla \alpha _{t_{o}}o\phi \right) ^{T}-\frac{\left\vert d\phi \right\vert
^{2}}{2}\nabla \alpha _{t_{o}}o\phi ,\overline{v}o\phi \right\rangle dv_{g}
\end{equation*}%
\begin{equation*}
-\int_{M}F^{\prime }\left( \frac{\left\vert d\phi \right\vert ^{2}}{2}%
\right) \left\langle d\phi \left( \nabla f_{t_{o}}\right) ,\overline{v}o\phi
\right\rangle dv_{g}
\end{equation*}%
and since ( see \cite{18} ) 
\begin{equation}
\nabla \alpha _{t_{o}}^{v}=-\frac{\left( \alpha _{t_{o}}^{v}\right) ^{2}}{%
\left\vert v\right\vert }sht_{o}\overline{v}  \label{3}
\end{equation}%
we have 
\begin{equation}
\left\langle \nabla \alpha _{t_{o}}o\phi ,\overline{v}o\phi \right\rangle =-%
\frac{\left( \alpha _{t_{o}}^{v}\right) ^{2}}{\left\vert v\right\vert }%
sht_{o}\left\vert \overline{v}o\phi \right\vert ^{2}\text{.}  \label{4}
\end{equation}%
Now let $\left( e_{1},...,e_{m}\right) $ be an orthogonal basis on $M$%
\begin{equation*}
\left\langle \left( \nabla \alpha _{t_{o}}o\phi \right) ^{T},\overline{v}%
o\phi \right\rangle =\sum_{i=1}^{m}\left\langle \nabla \alpha _{t_{o}}o\phi
,d\phi (e_{i})\right\rangle \left\langle \overline{v}o\phi ,d\phi
(e_{i})\right\rangle
\end{equation*}%
\begin{equation*}
=-\frac{sht_{o}}{\left\vert v\right\vert }\left( \alpha _{t_{o}}o\phi
\right) ^{2}\sum_{i=1}^{m}\left\langle \overline{v}o\phi ,d\phi
(e_{i})\right\rangle ^{2}
\end{equation*}%
\begin{equation*}
=-\frac{sht_{o}}{\left\vert v\right\vert }\left( \alpha _{t_{o}}o\phi
\right) ^{2}\left\vert d\phi _{v}\right\vert ^{2}\text{.}
\end{equation*}

Hence 
\begin{equation*}
\frac{d}{dt}E_{F}\left( \gamma _{t}^{v}o\phi \right) _{t=t_{o}}=
\end{equation*}%
\begin{equation*}
-2\frac{sht_{o}}{\left\vert v\right\vert }\int_{M}\alpha _{t_{o}}^{3}o\phi
.F^{\prime }\left( \alpha _{t_{o}}^{2}o\phi .\frac{\left\vert d\phi
\right\vert ^{2}}{2}\right) \left( \frac{\left\vert d\phi \right\vert ^{2}}{2%
}\left\vert \overline{v}o\phi \right\vert ^{2}-\left\vert d\phi
_{v}\right\vert ^{2}\right) dv_{g}
\end{equation*}%
\begin{equation*}
-\int_{M}\alpha _{t_{o}}^{2}o\phi .F^{\prime }\left( \frac{\left\vert d\phi
\right\vert ^{2}}{2}\right) \left\langle d\phi \left( \nabla
f_{t_{o}}\right) ,\overline{v}o\phi \right\rangle dv_{g}\text{.}
\end{equation*}
\end{proof}

We set 
\begin{equation*}
g(t)=\int_{M}\alpha _{t}^{2}o\phi .F^{\prime }\left( \frac{\left\vert d\phi
\right\vert ^{2}}{2}\right) \left\langle d\phi \left( \nabla f_{t}\right) ,%
\overline{v}o\phi \right\rangle dv_{g}\text{.}
\end{equation*}

\begin{lemma}
\label{lem3} 
\begin{equation*}
g(t)=
\end{equation*}%
\begin{equation*}
\int_{M}\alpha _{t}^{3}o\phi .F^{\prime \prime }(\alpha _{t}^{2}o\phi .\frac{%
\left\vert d\phi \right\vert ^{2}}{2})\left\vert d\phi \right\vert
^{2}\left\langle d\phi \left( \nabla \left( \alpha _{t_{o}}o\phi \right)
\right) ,\overline{v}o\phi \right\rangle dv_{g}
\end{equation*}%
\begin{equation*}
+\int_{M}\alpha _{t}^{2}o\phi .\left( F^{\prime \prime }(\alpha
_{t}^{2}o\phi .\frac{\left\vert d\phi \right\vert ^{2}}{2})\alpha
_{t}^{2}o\phi -\frac{F^{\prime }(\alpha _{t}^{2}o\phi .\frac{\left\vert
d\phi \right\vert ^{2}}{2})}{F^{\prime }\left( \frac{\left\vert d\phi
\right\vert ^{2}}{2}\right) }F^{\prime \prime }\left( \frac{\left\vert d\phi
\right\vert ^{2}}{2}\right) \right) \frac{\phi _{v}}{\left\vert v\right\vert 
}\left\vert d\phi \right\vert ^{2}dv_{g}\text{.}
\end{equation*}
\end{lemma}

\begin{proof}
First, we compute $\nabla f_{t}$%
\begin{equation*}
\nabla f_{t}=\frac{F^{\prime \prime }(\alpha _{t}^{2}o\phi .\frac{\left\vert
d\phi \right\vert ^{2}}{2})}{F^{\prime }\left( \frac{\left\vert d\phi
\right\vert ^{2}}{2}\right) }\left( \alpha _{t}o\phi .\left\vert d\phi
\right\vert ^{2}\nabla \left( \alpha _{t}o\phi \right) +\alpha _{t}^{2}o\phi
.\left\langle \nabla d\phi ,d\phi \right\rangle \right)
\end{equation*}%
\begin{equation*}
-\frac{F^{\prime }(\alpha _{t}^{2}o\phi .\frac{\left\vert d\phi \right\vert
^{2}}{2})}{F^{\prime }\left( \frac{\left\vert d\phi \right\vert ^{2}}{2}%
\right) ^{2}}F^{\prime \prime }\left( \frac{\left\vert d\phi \right\vert ^{2}%
}{2}\right) \left\langle \nabla d\phi ,d\phi \right\rangle
\end{equation*}%
\begin{equation*}
=\frac{F^{\prime \prime }(\alpha _{t}^{2}o\phi .\frac{\left\vert d\phi
\right\vert ^{2}}{2})}{F^{\prime }\left( \frac{\left\vert d\phi \right\vert
^{2}}{2}\right) }\alpha _{t}o\phi \left\vert d\phi \right\vert ^{2}\nabla
\left( \alpha _{t}o\phi \right)
\end{equation*}%
\begin{equation*}
+\left( \frac{F^{\prime \prime }(\alpha _{t}^{2}o\phi .\frac{\left\vert
d\phi \right\vert ^{2}}{2})}{F^{\prime }\left( \frac{\left\vert d\phi
\right\vert ^{2}}{2}\right) }\alpha _{t}^{2}o\phi -\frac{F^{\prime }(\alpha
_{t}^{2}o\phi .\frac{\left\vert d\phi \right\vert ^{2}}{2})}{F^{\prime
}\left( \frac{\left\vert d\phi \right\vert ^{2}}{2}\right) ^{2}}F^{\prime
\prime }\left( \frac{\left\vert d\phi \right\vert ^{2}}{2}\right) \right)
\left\langle \nabla d\phi ,d\phi \right\rangle \text{.}
\end{equation*}%
Then 
\begin{equation*}
g(t)=\int_{M}\alpha _{t_{o}}^{2}o\phi .F^{\prime }\left( \frac{\left\vert
d\phi \right\vert ^{2}}{2}\right) \left\langle d\phi \left( \nabla
f_{t_{o}}\right) ,\overline{v}o\phi \right\rangle dv_{g}
\end{equation*}%
\begin{equation*}
=\int_{M}\alpha _{t_{o}}^{3}o\phi .F^{\prime \prime }(\alpha
_{t_{o}}^{2}o\phi .\frac{\left\vert d\phi \right\vert ^{2}}{2})\left\vert
d\phi \right\vert ^{2}\left\langle d\phi \left( \nabla \left( \alpha
_{t_{o}}o\phi \right) \right) ,\overline{v}o\phi \right\rangle dv_{g}
\end{equation*}%
\begin{equation}
+\int_{M}\alpha _{t_{o}}^{2}o\phi .\left( F^{\prime \prime }(\alpha
_{t_{o}}^{2}o\phi .\frac{\left\vert d\phi \right\vert ^{2}}{2})\alpha
_{t_{o}}^{2}o\phi -\frac{F^{\prime }(\alpha _{t_{o}}^{2}o\phi .\frac{%
\left\vert d\phi \right\vert ^{2}}{2})}{F^{\prime }\left( \frac{\left\vert
d\phi \right\vert ^{2}}{2}\right) }F^{\prime \prime }\left( \frac{\left\vert
d\phi \right\vert ^{2}}{2}\right) \right)  \label{3''}
\end{equation}%
\begin{equation*}
\times \left\langle d\phi \left( \nabla \left\vert d\phi \right\vert
^{2}\right) ,\overline{v}o\phi \right\rangle dv_{g}\text{.}
\end{equation*}%
Let $\left\{ e_{1},...,e_{m}\right\} $ be a basis of $T_{x}M$ which
diagonalizes $\phi ^{\ast }can$, we have 
\begin{equation*}
\left\langle d\phi \left( \nabla \frac{\left\vert d\phi \right\vert ^{2}}{2}%
\right) ,\overline{v}o\phi \right\rangle =\left\langle \nabla _{e_{i}}d\phi
,d\phi \right\rangle \left\langle \overline{v}o\phi ,d\phi
(e_{j})\right\rangle \left\langle d\phi (e_{i}),d\phi (e_{j})\right\rangle
\end{equation*}%
\begin{equation*}
=\left\langle \nabla _{e_{i}}d\phi ,d\phi \right\rangle \left\langle 
\overline{v}o\phi ,d\phi (e_{j})\right\rangle \phi ^{\ast }can\left(
e_{i},e_{j}\right)
\end{equation*}%
\begin{equation*}
=\left\langle \nabla _{\overline{v}o\phi }d\phi \left( e_{j}\right) ,d\phi
\left( e_{j}\right) \right\rangle
\end{equation*}%
\begin{equation*}
=\left\langle \nabla _{d\phi \left( e_{j}\right) }\overline{v}o\phi ,d\phi
\left( e_{j}\right) \right\rangle +\left\langle \left[ \overline{v}o\phi
,d\phi \left( e_{j}\right) \right] ,d\phi \left( e_{j}\right) \right\rangle 
\text{.}
\end{equation*}%
Likewise we get%
\begin{equation*}
\left\langle \left[ \overline{v}o\phi ,d\phi \left( e_{j}\right) \right]
,d\phi \left( e_{j}\right) \right\rangle =-\frac{1}{2}\frac{d}{dt}\mid
_{t=0}\gamma _{t}^{\ast }\left\vert d\phi \left( e_{j}\right) \right\vert
^{2}
\end{equation*}%
\begin{equation*}
=-\frac{1}{2}\frac{d}{dt}\mid _{t=0}\alpha _{t}^{2}\left\vert d\phi \left(
e_{j}\right) \right\vert ^{2}
\end{equation*}%
and taking account of (\ref{1}) we obtain that 
\begin{equation*}
\left\langle \left[ \overline{v}o\phi ,d\phi \left( e_{j}\right) \right]
,d\phi \left( e_{j}\right) \right\rangle =\frac{\phi _{v}}{\left\vert
v\right\vert }\left\vert d\phi \left( e_{j}\right) \right\vert ^{2}
\end{equation*}%
so we infer that 
\begin{equation*}
\left\langle d\phi \left( \nabla \frac{\left\vert d\phi \right\vert ^{2}}{2}%
\right) ,\overline{v}o\phi \right\rangle =\frac{\phi _{v}}{\left\vert
v\right\vert }\left\vert d\phi \right\vert ^{2}+\left\langle \nabla _{e_{j}}%
\overline{v}o\phi ,d\phi \left( e_{j}\right) \right\rangle
\end{equation*}%
and 
\begin{equation*}
\left\langle \nabla _{e_{j}}\overline{v}o\phi ,d\phi \left( e_{j}\right)
\right\rangle =\nabla _{e_{j}}\left\langle \overline{v}o\phi ,d\phi \left(
e_{j}\right) \right\rangle -\left\langle \overline{v}o\phi ,\nabla
_{e_{j}}d\phi \left( e_{j}\right) \right\rangle
\end{equation*}%
\begin{equation*}
=\nabla _{e_{j}}\left\langle v,d\phi \left( e_{j}\right) \right\rangle
-\left\langle \overline{v}o\phi ,\nabla _{e_{j}}d\phi \left( e_{j}\right)
\right\rangle
\end{equation*}%
\begin{equation*}
=\left\langle v,\nabla _{e_{j}}d\phi \left( e_{j}\right) \right\rangle
-\left\langle \overline{v}o\phi ,\nabla _{e_{j}}d\phi \left( e_{j}\right)
\right\rangle
\end{equation*}%
\begin{equation*}
=\left\langle v-\overline{v}o\phi ,\nabla _{e_{j}}d\phi \left( e_{j}\right)
\right\rangle =0\text{.}
\end{equation*}%
Hence%
\begin{equation}
\left\langle d\phi \left( \nabla \frac{\left\vert d\phi \right\vert ^{2}}{2}%
\right) ,\overline{v}o\phi \right\rangle =\frac{\phi _{v}}{\left\vert
v\right\vert }\left\vert d\phi \right\vert ^{2}\text{.}  \label{4'}
\end{equation}
\end{proof}

Now set 
\begin{equation*}
\varphi (t_{o})=2\frac{sht_{o}}{\left\vert v\right\vert }\int_{M}\alpha
_{t_{o}}^{3}o\phi .F^{\prime }\left( \alpha _{t_{o}}^{2}o\phi .\frac{%
\left\vert d\phi \right\vert ^{2}}{2}\right) \left( -\frac{\left\vert d\phi
\right\vert ^{2}}{2}\left\vert \overline{v}o\phi \right\vert ^{2}+\left\vert
d\phi _{v}\right\vert ^{2}\right) dv_{g}
\end{equation*}%
\begin{equation*}
-\int_{M}\alpha _{t_{o}}^{3}o\phi .F^{\prime \prime }(\alpha
_{t_{o}}^{2}o\phi .\frac{\left\vert d\phi \right\vert ^{2}}{2})\frac{%
\left\vert d\phi \right\vert ^{2}}{2}\left\langle d\phi \left( \nabla \left(
\alpha _{t_{o}}o\phi \right) \right) ,\overline{v}o\phi \right\rangle dv_{g}
\end{equation*}%
and since by (\ref{3}) we have $_{{}}$%
\begin{equation*}
\left\langle d\phi \left( \nabla \left( \alpha _{t_{o}}o\phi \right) \right)
,\overline{v}o\phi \right\rangle =-\frac{sht_{o}}{\left\vert v\right\vert }%
\alpha _{t_{o}}^{2}o\phi \left\vert d\phi _{v}\right\vert ^{2}
\end{equation*}%
we get%
\begin{equation}
\varphi (t_{o})=2\frac{sht_{o}}{\left\vert v\right\vert }\int_{M}\alpha
_{t_{o}}^{3}o\phi .\left[ \left( F^{\prime }\left( \alpha _{t_{o}}^{2}o\phi .%
\frac{\left\vert d\phi \right\vert ^{2}}{2}\right) +\alpha _{t_{o}}^{2}o\phi
.\frac{\left\vert d\phi \right\vert ^{2}}{2}F^{\prime \prime }\left( \alpha
_{t_{o}}^{2}o\phi .\frac{\left\vert d\phi \right\vert ^{2}}{2}\right)
\right) \left\vert d\phi _{v}\right\vert ^{2}\right.  \label{5}
\end{equation}%
\begin{equation*}
\left. -F^{\prime }\left( \alpha _{t_{o}}^{2}o\phi .\frac{\left\vert d\phi
\right\vert ^{2}}{2}\right) \frac{\left\vert d\phi \right\vert ^{2}}{2}%
\left\vert \overline{v}o\phi \right\vert ^{2}\right] dv_{g}\text{.}
\end{equation*}%
or 
\begin{equation*}
\varphi (t_{o})=-2\frac{sht_{o}}{\left\vert v\right\vert }\int_{M}\alpha
_{t_{o}}^{3}o\phi .S_{g}^{F}\left( \gamma _{t}^{v}o\phi \right) dv_{g}\text{.%
}
\end{equation*}

\begin{proof}
( of Theorem \ref{th2}) \ Recall ( see \cite{13} ) that for any conformal
diffeomorphism $\gamma $ of the unit sphere $S^{n}$ there exist an isometry $%
r\in O\left( n+1\right) $, a real number $t\geq 0$ and a vector $v\in
R^{n+1}-\left\{ 0\right\} $ such that $\gamma =ro\gamma _{t}^{v}$, so it
suffices to consider $\gamma _{t}^{v}$ with $t\geq 0$ and $v\in
R^{n+1}-\left\{ 0\right\} $.

On the other hand 
\begin{equation*}
\frac{d}{dt}E_{F}\left( \gamma _{t}^{v}o\phi \right) =\varphi (t)+\chi (t)
\end{equation*}%
where%
\begin{equation*}
\chi (t)=-\int_{M}\alpha _{t}^{2}o\phi .\left( F^{\prime \prime }(\alpha
_{t}^{2}o\phi .\frac{\left\vert d\phi \right\vert ^{2}}{2})\alpha
_{t}^{2}o\phi -\frac{F^{\prime }(\alpha _{t}^{2}o\phi .\frac{\left\vert
d\phi \right\vert ^{2}}{2})}{F^{\prime }\left( \frac{\left\vert d\phi
\right\vert ^{2}}{2}\right) }F^{\prime \prime }\left( \frac{\left\vert d\phi
\right\vert ^{2}}{2}\right) \right) \frac{\phi _{v}}{\left\vert v\right\vert 
}\left\vert d\phi \right\vert ^{2}dv_{g}
\end{equation*}%
and $\varphi \left( t\right) $ is given by (\ref{5}). Now, since the
function $F$ is admissible we infer that$\ \chi (t)\leq 0$ . Since the $F$
energy-stress tensor $S_{g}^{F}\left( \phi \right) $of $\phi $ is positive (
resp. positive defined ) by assumption and

\begin{equation*}
\geq S_{g}^{F}\left( \phi \right)
\end{equation*}%
so the tensor field $S_{g}^{F}\left( \gamma _{t}o\phi \right) $ is positive
( resp. positive defined ). Consequently $\varphi (t)\leq 0$ ( resp. $%
\varphi (t)<0$ ) for any $\ t\geq 0$ and the proof of Theorem \ref{th2} is
complete.
\end{proof}

\end{document}